\documentclass{birkjour}
\usepackage{amsfonts,epsf,amsmath,amssymb,graphicx}
\usepackage{epstopdf}
\usepackage{color}
\def\GZZ{\mathbb{GZ}}
\def\QQ{\mathbb{Q}}
\def\GQQ{\mathbb{GQ}}

 \newtheorem{thm}{Theorem}[section]
 
 \newtheorem{lem}[thm]{Lemma}
 \newtheorem{claim}[thm]{Claim}
 
 \theoremstyle{definition}
 
 \theoremstyle{remark}

 \numberwithin{equation}{section}

\begin{document}

%
%
%
%
%
%
%
%
%

\title[Equivalence of the Generalized Zhang-Zhang Polynomial...]
 {Equivalence of the Generalized Zhang-Zhang Polynomial and the Generalized Cube Polynomial}

\author[Petra \v Zigert Pleter\v sek]{Petra \v Zigert Pleter\v sek}

\address{%
Faculty of Chemistry and Chemical Engineering, University of Maribor\\
Smetanova ulica 17, 2000 Maribor\\
Faculty of Natural Sciences and Mathematics, University of Maribor\\
Koro\v ska cesta 160, 2000 Maribor\\
Slovenia}

\email{petra.zigert@um.si}



\date{\today}

\begin{abstract}
In this paper we study the resonance graphs of benzenoid systems, tubulenes, and fullerenes. The resonance graph reflects the interactions between the Kekul\' e structures of a molecule. The equivalence of the Zhang-Zhang polynomial (which counts Clar covers) of the molecular graph and the cube polynomial (which counts hypercubes) of its resonance graph is known for all three families of molecular graphs. 

Instead of considering only interactions between 6-cycles (Clar covers), we also consider 10-cycles, which  contribute to the resonance energy of a molecule as well. Therefore, we generalize the concepts of the Zhang-Zhang polynomial and the cube polynomial and prove the equality of these two polynomials.
\end{abstract}

\maketitle
\section{Introduction}

 A benzenoid system is determined with all the hexagons lying inside cycle $C$ of the hexagonal lattice. They represent molecules called benzenoid hydrocarbons. These graphs are also known as the hexagonal systems and form one of the most extensively studied family of chemical graphs. For fundamental properties of benzenoid systems see \cite{gucy-89}, while some recent results can be found in \cite{kelenc,kovic-2014,trat,zhang-2014}. If we embed  benzenoid systems on a surface of a cylinder and join some edges we obtain structures called open-ended  single-walled carbon  nanotubes also called tubulenes. Carbon nanotubes are carbon compounds with a cylindrical structure and they were first observed in 1991  \cite{ii}. If we close a carbon nanotube with two caps composed of pentagons and hexagons, we obtain a fullerene. More exactly, a fullerene is a molecule of carbon in the form of a hollow sphere, ellipsoid, tube, or many other shapes. The first fullerene molecule was discovered $30$ years ago. In graph theory, a fullerene is a $3$-regular plane graph consisting only of pentagonal and hexagonal faces. The overview of some results on fullerene graphs can be found in \cite{andova}. Papers ~\cite{faghani,shi} present a sample of recent investigations.

The concept of the resonance graph appears quite naturally in the study of 
perfect matchings of molecular graphs of hydrocabons that represent Kekul\' e structures 
of corresponding hydrocarbon molecules. Therefore,
it is not surprising that it has been independently introduced in 
the chemical \cite{elba-93,grun-82}
as well as in the mathematical literature \cite{zhgu-88} (under the name
 $Z$-transformation graph) and then later rediscovered in \cite{rakl-96,rand-97}.

The equivalence of the Zhang-Zhang polynomial of the molecular graph and the cube polynomial of its resonance graph was established for benzenoid systems \cite{zhang-13}, tubulenes \cite{be-tr-zi}, and fullerenes \cite{tr-zi-2}. The Zhang-Zhang polynomial counts Clar covers with given number of hexagons, i.e.\,conjugated 6-cycles. For some recent research on the Zhang-Zhang polynomial see \cite{chou-2016,chou-2014} The resonance energy  is a theoretical quantity which is used for predicting the aromatic stability of conjugated systems. In the conjugated-circuit model, the resonance energy is determined with conjugated cycles of different lengths (see \cite{plavsic}), not only with 6-cycles. Among them, only 6-cycles and 10-cycles have uniquely determined structure. Therefore, we introduce the concept of the generalized Zhang-Zhang polynomial, which considers both of them. In this paper we prove the equivalence of the generalized Zhang-Zhang polynomial of a molecular graph and the generalized cube polynomial of the corresponding resonance graph. 

\section{Preliminaries}

A {\it benzenoid system} consists of a cycle $C$ of the infinite hexagonal lattice together with all hexagons inside $C$. A { \it benzenoid graph} is the underlying graph of a benzenoid system. 
\bigskip

\noindent
Next we formally define open-ended carbon nanotubes, also called {\em tubulenes} \cite{sa}. Choose any lattice point in the hexagonal lattice as the origin $O$. Let $\overrightarrow{a_1}$ and $\overrightarrow{a_2}$ be the two basic lattice vectors.
 Choose a vector $ \overrightarrow{OA} =n\overrightarrow{a_1}+m \overrightarrow{a_2}$
such that $n$ and $m$ are two integers and $|n|+|m|>1$, $nm\neq -1$. Draw two straight lines $L_1$ and $L_2$ passing through
$O$ and $A$ perpendicular to $O A$, respectively. By rolling up the hexagonal strip between $L_1$ and $L_2$ and gluing $L_1$ and $L_2$ such
that $A$ and $O$ superimpose, we can obtain a hexagonal tessellation $ \mathcal{HT}$ of the cylinder. $L_1$ and $L_2$ indicate the direction of
the axis of the cylinder. Using the terminology of graph theory, a {\em tubulene} $T$ is defined to be the finite graph induced by all
the hexagons of ${\mathcal HT}$ that lie between $c_1$ and $c_2$, where $c_1$ and $c_2$ are two vertex-disjoint cycles of ${\mathcal HT}$ encircling the axis of
the cylinder.  The vector $\overrightarrow{OA}$ is called the {\em chiral vector} of $T$ and  the cycles $c_1$ and $c_2$ are the two open-ends of $T$. 

\begin{figure}[!htb]
	\centering
		\includegraphics[scale=0.5, trim=0cm 0cm 1cm 0cm]{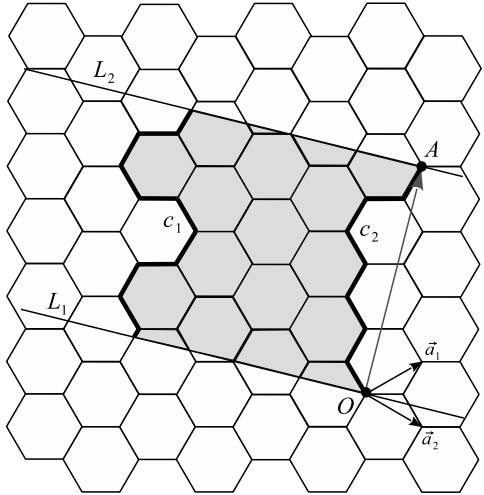}
\caption{Illustration of a $(4,-3)$-type tubulene.}
	\label{fig-nano}
\end{figure}

For any  tubulene $T$, if its chiral vector is $ n \overrightarrow{a_1} + m \overrightarrow{a_2}$, $T$ will be called an $(n,m)$-type tubulene, see Figure \ref{fig-nano}.
\bigskip

\noindent
A \textit{fullerene} $G$ is a {\color{black} $3$-connected 3-regular plane graph such that every face is bounded by either a pentagon or a hexagon}. By Euler's formula, it follows that 
the number of pentagonal faces of a fullerene is exactly $12$. 
\bigskip

\noindent
A {\em 1-factor} of a graph $G$ is a
spanning subgraph of $G$ such that every vertex has degree one. The edge set of a 1-factor is called a {\em perfect matching} of $G$, which is a set of independent edges covering all vertices of $G$. In chemical literature, perfect matchings are known as Kekul\'e structures (see \cite{gucy-89} for more details). 
Petersen's theorem states that {\color{black} every bridgeless $3$-regular graph always has a perfect matching} \cite{petersen}. Therefore, a fullerene always has at least one perfect matching. A hexagon of $G$ with exactly 3 edges in a perfect matching $M$ of $G$ is called a \textit{sextet}.

\bigskip

\noindent
Let $G$ be a benzenoid system, a tubulene or a fullerene with a perfect matching. The {\em resonance graph} $R(G)$ is the graph whose vertices are the  perfect matchings of $G$, and two perfect matchings are adjacent whenever
their symmetric difference forms a hexagon  of $G$. 
\bigskip

\noindent
The {\em hypercube} $Q_n$ of dimension $n$ is defined in the following way: 
all vertices of $Q_n$ are presented as $n$-tuples $(x_1,x_2,\ldots,x_n)$ where $x_i \in \{0,1\}$ for each $1\leq i\leq n$ 
and two vertices of $Q_n$ are adjacent if the corresponding $n$-tuples differ in precisely one coordinate.
\bigskip

\noindent A {\it convex subgraph} $H$ of a graph $G$ is a subgraph of $G$ such that every shortest path between two vertices of $H$ is contained in $H$.

\section{The generalized polynomials}

Let $G$ be a benzenoid system, a tubulene or a fullerene. A {\em Clar cover} is a spanning subgraph of $G$ such that every component of it is either $C_6$ or $K_2$. The {\em Zhang-Zhang polynomial}  of $G$ is defined in the following way:
$$ZZ(G,x)=\sum_{k \geq 0}z(G,k)x^k,$$
where $z(G,k)$ is the number of Clar covers of $G$ with $k$ hexagons.
\bigskip

\noindent
A {\em  generalized Clar cover} is a spanning subgraph of $G$ such that every component of it is either $C_6$, $C_{10}$ or $K_2$. See Figure \ref{clar_cover} for an example.

\begin{figure}[!htb]
	\centering
		\includegraphics[scale=0.7, trim=0cm 0cm 0cm 0cm]{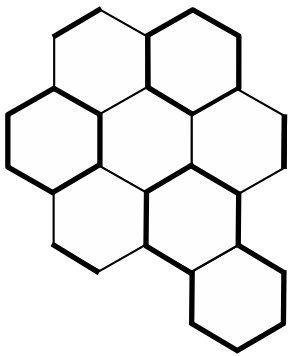}
\caption{A generalized Clar cover of a benzenoid system $G$.}
	\label{clar_cover}
\end{figure}

The {\em generalized Zhang-Zhang polynomial}  of $G$ is defined in the following way:
$$GZZ(G,x,y)=\sum_{k \geq 0, l \geq 0}gz(G,k,l)x^ky^l,$$
where $gz(G,k,l)$ is the number of generalized Clar covers of $G$ with $k$ cycles $C_{6}$ and $l$ cycles $C_{10}$.
 Note that for a graph $G$ 
number $gz(G,0,0)$ equals the number of  vertices of $R(G)$ and $gz(G, 1,0)$ equals the number of edges of $R(G)$. Furthermore, number $gz(G,k,0)$ represents the number of Clar covers with $k$ hexagons.
\bigskip

\noindent
Let $H$ be a graph. 
The {\em Cube polynomial} of $H$ is defined  as follows:
$$C(H,x)=\sum_{k\geq 0} \alpha_{k} (H)x^k,$$
where $\alpha_{k}(H)$ denotes the number of induced subgraphs of $H$ that are isomorphic to the $k$-dimensional hypercube.
\bigskip

\noindent
Let $G$ be a graph and $i \geq 1$ an integer. Then by $G^i$ we denote the Cartesian product of $i$ copies of $G$, i.e. $G^i = G \Box \cdots \Box G$. Also, $G^0 = K_1$. Furthermore, for any $k,l \geq 0$ we define $Q_{k,l} = P_2^k \Box P_3^l$, where $P_2$ and $P_3$ are paths on 2 and 3 vertices, respectively. Obviously, $Q_{k,0}$ is the $k$-dimensional hypercube. Moreover, if $k+l > 0$, vertices of the graph $Q_{k,l}$ can be presented as $(k+l)$-tuples $(b_1, \ldots, b_k, b_{k+1}, \ldots, b_{k+l})$, where $b_i \in \lbrace 0,1\rbrace$ if $i \in \lbrace 1, \ldots, k \rbrace$ and $b_i \in \lbrace 0,1, 2\rbrace$ if $i \in \lbrace k+1, \ldots, k+l \rbrace$. In such representation two vertices $(b_1, \ldots, b_k, b_{k+1}, \ldots, b_{k+l})$ and $(b_1', \ldots, b_k', b_{k+1}', \ldots, b_{k+l}')$ are adjacent if and only if there is $i \in \lbrace 1, \ldots, k+l\rbrace$ such that $|b_i - b_i'|= 1$ and $b_j = b_j'$ for any $j \neq i$.

\noindent
Let $H$ be a graph. 
The {\em generalized Cube polynomial} of $H$ is defined  as follows:
$$GC(H,x,y)=\sum_{k\geq 0, l\geq 0} \alpha_{k,l} (H)x^ky^l,$$
where $\alpha_{k,l}(H)$ denotes the number of induced convex subgraphs of $H$ that are isomorphic to the graph $Q_{k,l}$.

\section{The main result}

In this section we prove that the generalized Zhang-Zhang polynomial of every benzenoid system, tubulene or fullerene equals the generalized cube polynomial of its resonance graph. 

\begin{thm}
Let $G$ be a benzenoid system, a tubulene or a fullerene with a perfect matching.  Then the  generalized Zhang-Zhang polynomial of $G$ equals the generalized cube polynomial 
of its resonance graph $R(G)$, i.e.
    $$GZZ(G,x,y)=GC(R(G),x,y)\,.$$
    \label{main}
\end{thm}

\begin{proof}

Let $k$ and $l$ be nonnegative integers. For a graph $G$ we denote by $\GZZ(G,k,l)$ the set of all generalized Clar covers of $G$ with exactly $k$ cycles $C_6$ and $l$ cycles $C_{10}$. On the other hand, consider a graph $H$;  the set of induced convex  subgraphs of $H$ that are isomorphic to a graph $Q_{k,l}$ is denoted by $\GQQ_{k,l}(H)$. Let us define a mapping $f_{k,l}$ from the set of generalized Clar covers of $G$ with $k$ cycles $C_6$ and $l$ cycles $C_{10}$ to the set of 
induced convex subgraphs of the resonance graph $R(G)$ isomorphic to the graph $Q_{k,l}$
$$f_{k,l}:\, \GZZ(G,k,l) \longrightarrow \GQQ_{k,l}(R(G))$$
in the following way:  for a generalized Clar cover  $C \in \GZZ(G,k,l) $ consider all perfect matchings
$M_1$, $M_2$, $\ldots$, $M_i$ of $G$ such that:
\begin{itemize}
\item  if cycle $C_6$ in $C$, then $|M_j \cap E(C_6)| =3$  for all $j=1,2, \ldots,i$,
\item if cycle $C_{10}$ of $C$ is composed of two hexagons, $h_1$ and $h_2$, then $|M_j \cap (E(h_1) \cup E(h_2))| = 5$ for all $j=1,2, \ldots,i$,
\item each isolated edge of $C$ is in $M_j$ for all $j=1,2, \ldots,i$.
\end{itemize}
Finally, assign $f_{k,l}(C)$  as an induced subgraph of $R(G)$ with vertices $M_1$, $M_2$,$\ldots$, $M_i$.

Note first that in case when  $k=0$ and $l=0$ generalized Clar covers  are the perfect matchings of $G$ and if $C$ is such a generalized Clar cover then $f_{k,l}(C)$ is a vertex of the resonance graph and the mapping is obviously bijective. So from now on at least one of $k$ and $l$ will be positive. We first show that $f_{k,l}$ is a well-defined function.

\begin{lem}
For each generalized Clar cover $C \in \GZZ(G,k,l) $ it follows that $f_{k,l}(C) \in \QQ_{k,l}(R(G))$.
\label{lema1}
\end{lem}
\begin{proof} First we show that $f_{k,l}(C)$ is isomorphic to the graph $Q_{k,l}$. Let
$c_1$, $c_2$,$\ldots$,$c_k$ be the hexagons of $C$ and let $c_{k+1}, \ldots, c_{k+l}$ be cycles $C_{10}$ that are in $C$. 
Obviously, every hexagon of $C$ has two possible perfect matchings. Let us call these ``possibility 0" and ``possibility 1". Moreover, for every cycle $C_{10}$ in $C$ we obtain tree possible perfect matchings of graph $\langle V(C_{10})\rangle$, which will be denoted as ``possibility 0", ``possibility 1", and ``possibility 2". Also, if cycle $C_{10}$ is composed of hexagons $h_1$ and $h_2$, ``possibility 1" denotes the perfect matching containing the common edge of $h_1$ and $h_2$. 

For any vertex $M$ of $f_{k,l}(C)$ let 
$b(M) =(b_1,b_2,\ldots, b_k, b_{k+1}, \ldots, b_{k+l})$, where $b_j=i$ if on $c_j$ possibility $i$ is selected.
It is obvious that $b:V(f_{k,l}(C))\rightarrow V(Q_{k,l})$ is a bijection. Let $b(M')=(b_1',b_2',\ldots, b_k', b_{k+1}', \ldots, b_{k+l}')$ for $M' \in V(f_{k,l}(C))$. If $M$ and $M'$ are adjacent in $f_{k,l}(C)$, then $M\oplus M'=E(h)$ for a hexagon $h$ of some $c_i$, where $1 \leq i \leq k+l$. Therefore, $b_j=b_j'$ for each $j \neq i$ and $|b_i - b_i'| = 1$, which implies
 that $b(M)$ and $b(M')$ are adjacent in $Q_{k,l}$. Conversely, if
  $(b_1,b_2,\ldots, b_k, b_{k+1}, \ldots, b_{k+l})$ and $(b_1',b_2',\ldots, b_k', b_{k+1}', \ldots, b_{k+l}')$ are adjacent in $Q_{k,l}$, it follows that $M$ and 
  $M'$ are adjacent in $f_{k,l}(C)$. Hence $b$ is an isomorphism between $f_{k,l}(C)$ and $Q_{k,l}$.
  
To complete the proof we have to show that $f_{k,l}(C)$ is a convex subgraph of $R(G)$. Therefore, let $M$ and $M'$ be two vertices of $f_{k,l}(C)$. Obviously, perfect matchings $M$ and $M'$ can differ only in the edges of hexagons that belong to cycles of $C$. Therefore, any shortest path between $M$ and $M'$ in $R(G)$ contains perfect matchings that are vertices of $f_{k,l}(C)$. It follows that $f_{k,l}(C)$ is convex in $R(G)$.
\end{proof}

\noindent
The following lemma shows that $f_{k,l}$ is injective.

\begin{lem}
The mapping  $f_{k,l}:\, \GZZ(G,k,l) \longrightarrow \QQ_{k,l}(R(G))$ is injective for any integers $k,l$.
\label{lema2}
\end{lem}
\begin{proof} Let $C$ and $C'$ be distinct generalized Clar covers in $\GZZ(G,k,l)$. If $C$ and $C'$ contain the same set of cycles, then the isolated edges of $C$ and $C'$ are distinct. Therefore, $f_{k,l}(C)$ and $f_{k,l}(C')$ are disjoint induced subgraphs of $R(G)$ and thus $f_{k,l}(C)\neq f_{k,l}(C')$. Therefore, suppose that $C$ and $C'$ contain different sets of cycles. Without loss of generality we can assume that there is hexagon $h$ such that $h$ has at least five edges in $C$ and $h$ has at most three edges in $C'$. Hence  at least one edge $e$ of $h$ does not belong to $C'$. From the definition of the function $f_{k,l}$, $e$ is thus unsaturated by those perfect matchings that correspond to the vertices in $f_{k,l}(C')$. However, there obviously exists perfect matching $M \in V(f_{k,l}(C))$ such that $e \in M$. As a result, $M \notin V (f_{k,l}(C'))$ and $f(C) \neq f(C')$.
\end{proof}

\noindent
The next lemma was proved in \cite{zhang-13} for benzenoid systems. The same proof can be applied in the case of tubulenes or fullerenes.
\begin{lem} \cite{zhang-13}
\label{stiri}
Let $G$ be a benzenoid systems, a tubulene, or a fullerene with a perfect matching. If the resonance graph $R(G)$ contains a 4-cycle $M_1M_2M_3M_4$,
then $h = M_1\oplus M_2$ and $h' = M_1\oplus M_4$ are disjoint hexagons. Also, we have $h = M_3\oplus M_4$
and $h' = M_2\oplus M_3$.
\end{lem}

\noindent
The following lemma shows that $f_{k,l}$ is surjective.

\begin{lem}
The mapping  $f_{k,l}:\, \GZZ(G,k,l) \longrightarrow \QQ_{k,l}(R(G))$ is surjective for any integers $k,l$.
\label{lema3}
\end{lem}

\begin{proof}
Let $k,l$ be integers and $Q \in \mathbb{Q}_{k,l}(R(G))$. Then the vertices of $Q$ can be identified with strings $(b_1, \ldots, b_k, b_{k+1}, \ldots, b_{k+l})$, where $b_i \in \lbrace 0,1\rbrace$ if $i \in \lbrace 1, \ldots, k \rbrace$ or $b_i \in \lbrace 0,1, 2\rbrace$ if $i \in \lbrace k+1, \ldots, k+l \rbrace$, so that two vertices of $Q$ are adjacent in $Q$ if and only if their strings $b$ and $b'$ differ in precisely one position $i$, such that $|b_i-b_i'|=1$. 

Let $M = (0,0,0, \ldots, 0)$, $N^1 = (1,0,0, \ldots, 0)$, $N^2 = (0,1,0, \ldots, 0)$, \ldots, $N^{k+l} = (0,0,0, \ldots, 1)$ be the vertices of $Q$. It is obvious that $MN^i$ is an edge of $R(G)$ for every $i, 1 \leq i \leq k+l$. By definition of $R(G)$, the symmetric difference of  perfect matchings  $M$ and $N^i$ is the edge set of a  hexagon of $G$. We denote this hexagon by $h_i$ and we obtain the set of hexagons $\lbrace h_1, \ldots, h_{k+l} \rbrace$ of graph $G$. If two of these hexagons were the same, for example if $h_i=h_j$ for $i,j \in \lbrace 1, \ldots, k+l \rbrace$ and $i \neq j$, then $N^i = N^j$ - a contradiction. Hence, we have the set of $k+l$ distinct hexagons. In the next claim we show that these hexagons are pairwise disjoint.

\begin{claim}
The hexagons $h_i$, $1 \leq i \leq k+l$, are pairwise disjoint.
\end{claim}
\begin{proof}
Let $i,j \in \lbrace 1, \ldots, k+l \rbrace$ and $i \neq j$. Let $W$ be a vertex of $Q$ having exactly two $1$'s (and these are in the $i$th and $j$th position) and $0$ at every other position. Obviously, $MN^iWN^j$ is a 4-cycle and therefore, by Lemma \ref{stiri}, $h_i$ and $h_j$ are disjoint hexagons. 
\end{proof}

Next, we consider the vertices $O^{i}$, $i \in \lbrace k+1, \ldots, k+l\rbrace$, such that $O^i$ has $2$ in the $i$th position and $0$ in every other position. Obviously, $N^iO^i$ is the edge of $R(G)$ for any $i \in \lbrace k+1, \ldots, k+l\rbrace$. Let $h_i'$ be the hexagon of $G$ corresponding to the edge $N^iO^i$.

\begin{claim}
If $i \in \lbrace k+1, \ldots, k+l\rbrace$, the hexagon $h_i'$ has exactly one common edge with $h_i$.
\end{claim}

\begin{proof}
It is easy to see that $h_i \neq h_i'$ (otherwise $M=O^i$). Therefore, suppose that $h_i$ and $h_i'$ are disjoint. Since they are both sextets in the perfect matching $N^i$, there is a vertex $X$ of $R(G)$, $X \neq N^i$, which is adjacent to $M$ and $O^i$. If $X \in V(Q)$, the string of $X$ must differ from $M$ for $1$ in exactly one position and must differ from $O^i$ for $1$ in exactly one position, which means $X = N^i$ - a contradiction. Therefore, $X$ is not in $Q$. Since $MXO^i$ is a shortest path between $M$ and $O^i$, $Q$ is not convex subgraph of $R(G)$, which is a contradiction. Hence, $h_i$ and $h_i'$ have exactly one common edge. 
\end{proof}

\begin{claim}
Let $i \in \lbrace k+1, \ldots, k+l\rbrace$. Then the hexagon $h_i'$ is disjoint with every $h_j$, $j \in \lbrace 1, \ldots, k+l \rbrace \setminus \lbrace i\rbrace$.
\end{claim}

\begin{proof}
Let $X$ be a vertex in $Q$ with $2$ in the $i$th position, $1$ in the $j$th position and $0$ in every other position. Furthermore, let $Y$ be a vertex in $Q$ with $1$ in the $i$th position, $1$ in the $j$th position and $0$ in every other position. Obviously, $h_i' \neq h_j$ (otherwise $O^i=Y$, which is a contradiction).  Since $N^iO^iXY$ is a $4$-cycle such that $h_i'$ corresponds to the edge $N^iO^i$ and $h_j$ corresponds to the edge $N^iY$, it follows from Lemma \ref{stiri} that hexagons $h_i'$ and $h_j$ are disjoint.

\end{proof}

\begin{claim}
Let $i \in \lbrace k+1, \ldots, k+l\rbrace$. Then the hexagon $h_i'$ is disjoint with every $h_j'$, $j \in \lbrace k+1, \ldots, k+l\rbrace \setminus \lbrace i \rbrace$.
\end{claim}

\begin{proof}
Define the following vertices in $Q$:
\begin{itemize}
\item $X_1$ has $1$ in the $i$th position, $1$ in the $j$th position and $0$ in every other position,
\item $X_2$ has $1$ in the $i$th position, $2$ in the $j$th position and $0$ in every other position,
\item $X_3$ has $2$ in the $i$th position, $1$ in the $j$th position and $0$ in every other position,
\item $X_4$ has $2$ in the $i$th position, $2$ in the $j$th position and $0$ in every other position.
\end{itemize}

Using Lemma \ref{stiri} we can easily see that hexagon $h_j'$ corresponds to the edge $X_1X_2$ and hexagon $h_i'$ corresponds to the edge $X_1X_3$. Since $X_1X_2X_4X_3$ is a $4$-cycle in the resonance graph, Lemma \ref{stiri} again implies that $h_i'$ and $h_j'$ are disjoint and the proof is complete. 
\end{proof}

Let $C_i$, $i \in \lbrace k+1, \ldots, k+l \rbrace$ be a $10$-cycle formed by $h_i$ and $h_i'$. Moreover, let $C$ be a spanning subgraph of $G$ such that $E(C) = M \cup  E(h_{1}) \cup \ldots \cup E(h_{k}) \cup E(C_{k+1}) \cup \ldots E(C_{k+l})$. Therefore, $C$ is a generalized Clar cover with $k$ hexagons and $l$ 10-cycles. It is obvious that every edge in $Q$ corresponds to some hexagon $h_i$, $i \in \lbrace 1, \ldots, k+l \rbrace$ or $h_i'$, $i \in \lbrace k+1, \ldots, k+l \rbrace$. Therefore, $V(f_{k,l}(C))=V(Q)$. Since both $Q$ and $f_{k,l}(C)$ are induced subgraphs of the resonance graph, it follows $f_{k,l}(C)=Q$. 
\end{proof}

We have proved that $f_{k,l}$ is bijective function and hence, $|\GZZ(G,k,l)| = |\GQQ_{k,l}(R(G))|$. Therefore, the proof is complete.

\end{proof}

\section{An example}

In this final section we give an example of a benzenoid system $G$ and calculate the generalized Zhang-Zhang polynomial of $G$, i.e. the generalized cube polynomial of the resonance graph of $G$. See Figures \ref{ben_sistem} and \ref{resonancni}. 

\begin{figure}[!htb]
	\centering
		\includegraphics[scale=0.7, trim=0cm 0cm 0cm 0cm]{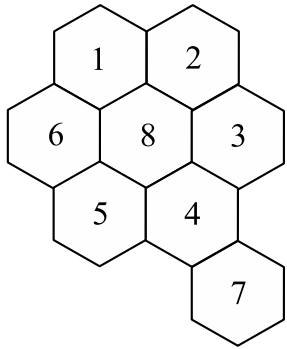}
\caption{Benzenoid system $G$.}
	\label{ben_sistem}
\end{figure}

\begin{figure}[!htb]
	\centering
		\includegraphics[scale=0.7, trim=0cm 0cm 0cm 0cm]{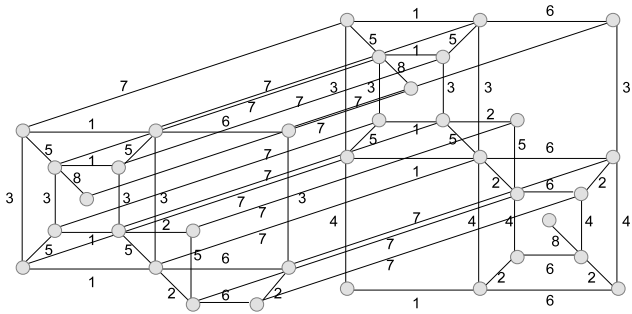}
\caption{Resonance graph $R(G)$.}
	\label{resonancni}
\end{figure}

The polynomials are
$$GZZ(G,x,y)=GC(R(G),x,y)= $$ $$=34+53x+35x^2+12x^3+x^4+48y+7y^2+37xy+xy^2+3x^2y\,.$$
For example, the coefficient in front of $x^2y$ is 3, since there are 3 generalized Clar covers in $G$ with two $C_6$ and one $C_{10}$. On the other hand, this coefficient counts the number of induced convex subgraphs of $R(G)$ isomorphic to the graph $P_2^2 \Box P_3$. 

\section*{Acknowledgment}

Supported in part by the Ministry of Science of Slovenia under grant $P1-0297$.


\end{document}